\documentclass[10pt]{article}
\usepackage{amssymb}
\usepackage{amsmath}
\usepackage{amsthm}
\usepackage[utf8]{inputenc}
\usepackage[T1]{fontenc}
\usepackage[pdfborder={0 0 0},bookmarks=false]{hyperref}
\newtheorem{thm}{Theorem}
\newtheorem{mainthm}{Theorem}
\newtheorem{cor}[thm]{Corollary}
\newtheorem{lem}[thm]{Lemma}
\newtheorem{defn}[thm]{Definition}
\newtheorem{prop}[thm]{Proposition}

\def\CompressMatrices{\ifmmode \def\quad{\hskip.5em\relax}\fi}

\def\ad{\operatorname{ad}}

\def\gen#1{\left\langle #1 \right\rangle}
\def\ring#1{\left\langle #1 \right\rangle}

\def\ann#1#2{\operatorname{ann}_{#1}(#2)}

\def\rightonto{\twoheadrightarrow}
\def\rightinto{\hookrightarrow}
\def\card{\operatorname{card}}

\begin{document}

\title{On commutativity of ideal extensions.}
\author{Joachim Jelisiejew}
\date{}
\maketitle

\begin{abstract}
    In this paper we examine the commutativity of ideal extensions. We introduce methods of
    constructing such extensions, in particular we construct a noncommutative
    ring $T$ which contains a~central and idempotent ideal $I$
    such that $T/I$ is a field. This answers a question from
    \cite{puczylowski}. Moreover we classify fields of characteristic $0$ which can be obtained as
    $T/I$ for some $T$.
\end{abstract}

    \vspace{0.7cm}
    {\small\noindent\textbf{e-mail address:} \texttt{jj277546@students.mimuw.edu.pl}\\
    \textbf{keywords:} ideal ring extension, derivation, semidirect
    extension, idempotent ring}

\section{Introduction}
\renewcommand{\themainthm}{\Alph{mainthm}}

All rings considered in this paper are associative but not necessarily with
unity. Throughout the paper we will use the following definition of (ideal) extension
\begin{defn}
    \emph{Ideal extension} of a given ring $I$ by a given ring $Q$ is a ring $R$,
    which contains a ideal $I'$ isomorphic (as a ring) to $I$ and such
    that $R/I' \simeq Q$.
\end{defn}

Ideal extensions of rings were studied in many papers. In \cite{sands, puczylowski}
the problem of commutativity of an ideal extension was studied in the case
when $I = I^2$ and both $I$ and $Q$ are commutative. It was proved there that
for a given commutative $I = I^2$, ideal extension of $I$ by arbitrary
commutative $Q$ is commutative iff $\ann{I}{I} := \{i\in I|\ iI = Ii = 0\} = 0$. In
\cite{puczylowski} it was also shown that if $Q$ is a field which is
algebraic over its prime subfield, then for arbitrary commutative $I = I^2$,
ideal extension of $I$ by $Q$ is commutative.
\def\sp{\emph{Idempotence problem}}%
\def\wsp{\emph{Centrality problem}}%
The problem of description of all fields for which this property holds was
left open. The main aim of this paper is to solve it.
Let us state explicitly
\begin{center}
    \emph{The }\sp{}. Describe fields $K$ such that for any commutative
    ring $I$ satisfying $I = I^2$ and any ideal extension $R$ of $I$ by $K$
    the ring $R$ is commutative.
\end{center}

Note that if $I = I^2$ is an ideal of $R$ and $I$ is commutative then $I$ is
contained in the center $Z(R)$ of $R$ (we will shortly say that $I$ is \emph{central}
in $R$). In this context it is natural to state
\begin{center}
    \emph{The }\wsp{}. Describe fields $K$ such that for any ring $R$ and
    central ideal $I$ of $R$ satisfying $R/I \simeq K$ the ring $R$ is
    commutative.
\end{center}
It turns out that the class of fields for the \wsp{} coincides with the class of fields
for the \sp{}.

In the first section of the article we introduce derivations and prove the following result,
which allows us to solve the \sp{} and the \wsp{} restricted to the fields of
characteristic zero

        \begin{mainthm}
            \label{maintwo}

            If $K$ is a field of characteristic $0$, then the following
            conditions are equivalent:
            \begin{enumerate}
                \item the transcendence degree of $K$ over $\mathbb{Q}$ is greater
                    than $1$,
                \item there exists two derivations from $K$ to $K$ whose kernels
                    are incomparable with respect to inclusion,
                \item there exists a field $L$ containing $K$ and two derivations
                    from $K$ to $L$ whose kernels
                    are incomparable with respect to inclusion.
            \end{enumerate}
\end{mainthm}

The characteristic zero assumption in Theorem \ref{maintwo} may be weakened --
see Corollary \ref{incomparabledontexists}.
\def\ext#1#2{T\left(#1, #2\right)}

In the two further sections of the paper we provide two constructions of ideal extensions.
The former of those allows us to solve the \wsp{}, but to solve the
\sp{} we need to use also the latter construction. The former construction
has applications not connected with
the subject of this article; a basic example is
given in the comments section.

Based on the constructions we prove the following main result of
the paper

\begin{mainthm}
    \label{mainone}
    For a given field $K$, the following conditions are equivalent
    \begin{enumerate}
        \item there exists a field $L$ containing $K$ and two derivations
            from $K$ to $L$ whose kernels
            are incomparable with respect to inclusion,
        \item there exists a noncommutative ring $R$ and its ideal $I$ such that $I$ is
            central in $R$, $R/I \simeq K$ and $I^2 = 0$,
        \item there exists a noncommutative ring $R$ and its ideal $I$ such that $I$ is
            commutative as a ring, $I =
            I^2$ and $R/I \simeq K$.
        \item there exists a noncommutative ring $R$ and its ideal $I$ such that $I$ is
            central in $R$ and $R/I \simeq K$.
    \end{enumerate}
\end{mainthm}

Joining theorems \ref{maintwo} and \ref{mainone} one obtains
\begin{cor}
    The class of fields for \sp{} coincides with the class of fields for
    \wsp{} and the class of fields $K$ such that there exist a field $L$
    containing $K$ and two derivations from $K$ to $L$ whose kernels are
    incomparable with respect to inclusion.

    When we restrict to the class of fields of characteristic zero then the
    above classes coincide with the class of fields having transcendence
    degree at most one over their prime subfields.
\end{cor}

\section{Derivations and their kernels}

    The structure of the derivations has many important connections with
    extensions. In this section we will develop the machinery used in the rest
    of the paper and prove Theorem \ref{maintwo}.

\begin{defn}
    If $M$ is a left $S$ module and a right $R$ module, then $M$ is a $(S,
    R)$--\emph{bimodule} iff $(s\cdot m)\cdot r = s\cdot (m \cdot r)$ for all $s\in
    S, m\in M, r\in R$. When dealing with bimodules we omit parentheses.

    If ring $R$ has a unity then a $(R, R)$--bimodule $M$ is \emph{unitary}
    iff $1\cdot m = m\cdot 1 = m$ for all $m\in M$.
\end{defn}

\begin{defn}
    A \emph{derivation} from a ring $R$ into a $(R, R)$--bimodule $M$
    is a map $d:R \to M$ satisfying
    \begin{enumerate}
        \item $d(x) + d(y) = d(x+y)$,
        \item $d(xy) = x\cdot d(y) + d(x)\cdot y$,
    \end{enumerate}
    for all $x, y\in R$.

    The set $R^d:=\left\{ r\in R\ |\ d(r) = 0 \right\}$ is a subring of
    $R$ referred to as the kernel of $d$.
\end{defn}

\begin{defn}
    Derivations $d_1, d_2:R\to M$ are called \emph{incomparable} if their
    kernels are incomparable with respect to inclusion, i.e., if
    there exist $x_1, x_2\in R$ such that
    \[
    d_1(x_1) \neq 0,\ \ d_1(x_2) = 0,\ \ d_2(x_1) = 0,\ \ d_2(x_2) \neq 0.
    \]
\end{defn}

\begin{lem}
    Let $K$ be a field and $d:K  \to M$ be a derivation. If $M$ is a
    unitary $(K,K)$--bimodule, then $K^d$ is a subfield of $K$.
\end{lem}
\begin{proof}
    We already know that $K^d$ is a subring of $K$.
    Note that $d(1) = d(1\cdot 1) = 1\cdot d(1) + d(1)\cdot 1 = d(1) + d(1)$, so $1\in
    K^d$.
    For every $0\neq x\in K$ we have $0 = d(1) = d(xx^{-1}) = x \cdot
    d(x^{-1}) + d(x)\cdot x^{-1}$, so if $d(x) = 0$ then $d(x^{-1}) = 0$.
    Consequently $K^d$ is closed with respect to taking inverses, so it is a
    subfield.
\end{proof}

\newpage
\begin{prop}
    \label{algebraically-closed}
    Let $K$ be a field and $M$ an unitary $(K, K)$--bimodule such that $k\cdot m = m\cdot k$ for all $k\in
    K, m\in M$.  If $d:K  \to M$ is a derivation such that every element of $K$
    which is algebraic over $K^d$ is separable then $K^d$ is algebraically closed in
    $K$.
\end{prop}
\begin{proof}
    Let $\alpha\in K$ be algebraic over $K^d$, $f\in K^d[x]$ be its minimal
    polynomial.
    By induction we have $d(\alpha^n) = n\alpha^{n-1}\cdot
    d(\alpha)$ for every positive integer $n$, hence
    \[
    0 = d(0) = d(f(\alpha)) = f'(\alpha)\cdot d(\alpha).
    \]
    Now $\alpha$ is separable, so $f'(\alpha) \neq 0$. Consequently
    $d(\alpha) = 1\cdot d(\alpha) =
    \left(f'(\alpha)\right)^{-1}\cdot(f'(\alpha)\cdot d(\alpha)) = 0$,
    $\alpha\in K^d$.
\end{proof}

\begin{cor}
    \label{incomparabledontexists}
    If $K$ is a field which is algebraic over its prime subfield or $K$ is
    of characteristic $0$ and the transcendence degree of $K$ over
    $\mathbb{Q}$ is smaller than $2$, then for any field $L$ containing $K$
    every two derivations from $K$ to $L$ are comparable.
\end{cor}
\begin{proof}
    The assumptions imply that every subfield of $K$ is perfect, so for any
    derivation $d:K\to L$ the separability condition from proposition \ref{algebraically-closed} is satisfied
    and we deduce that $K^d$ is algebraically closed in $K$.

    Since the transcendence degree of $K$ over its
    prime field is at most $1$, there are at most two algebraically closed
    (in $K$) subfields of $K$ (the closure of the prime subfield and $K$
    itself) so
    any two derivations are comparable.
\end{proof}

    \setcounter{mainthm}{0}
    \begin{mainthm}
        If $K$ is a field of characteristic $0$, then the following
        conditions are equivalent:
        \begin{enumerate}
            \item\label{maintwo1} the transcendence degree of $K$ over $\mathbb{Q}$ is greater
                than $1$,
            \item\label{maintwo2} there exists two derivations from $K$ to $K$ whose kernels
                are incomparable with respect to inclusion,
            \item\label{maintwo3} there exists a field $L$ containing $K$ and two derivations
                from $K$ to $L$ whose kernels
                are incomparable with respect to inclusion,
        \end{enumerate}
\end{mainthm}

\begin{proof}
    $\ref{maintwo1}.\implies\ref{maintwo2}.$
    It is known (e.g. \cite{zariski}) that if $L' \subseteq L$ is an algebraic extension of fields of characteristic
    $0$, then every derivation $d:L'\to L'$ extends to a derivation $\hat{d}: L
    \to L$.

    Let $K' \subseteq K$ be a purely transcendental extension of
    $\mathbb{Q}$ such that $K' \subseteq K$ is algebraic. Since $K'$ is
    isomorphic to the field
    of rational functions over $\mathbb{Q}$ in at least two variables,
    there exist two incomparable derivations $d_1, d_2:K'\to K'$ (e.g. the partial
    derivations associated to two different variables). The desired derivations
    are extensions of $d_1, d_2$ to derivations from $K$ to $K$.

    $\ref{maintwo2}\implies\ref{maintwo3}$. Obvious.

    $\ref{maintwo3}\implies\ref{maintwo1}$. Direct consequence of corollary \ref{incomparabledontexists}.
\end{proof}

\def\M{\mathbb{M}}
\def\m{\mathfrak{m}}
\section{Construction 1}

 The following easy to check proposition
contains a method of building a certain ring from a given
ring, subring and two derivations.
\begin{prop}
    \label{mainconstruction}
    \label{extconstruction}
    Let $R$ be a ring, $S$ be its subring and $d_1,d_2:S\to R$ be derivations
    with respect to the natural $(S,S)$-bimodule structure on $R$. Then
    \begin{enumerate}
        \item   \[
            \ext{S}{R}:=\left\{ \begin{pmatrix}f & d_1(f) & g\\0 & f & d_2(f)\\ 0 & 0 & f\end{pmatrix}\
                    \Big|\ f\in S, g \in R\right\} \subseteq \M_3(R)
                \]
                is a subring of $\M_3(R)$,
        \item \[\m = \left\{ \begin{pmatrix}0 & 0 & g\\0 & 0 & 0\\ 0 & 0 & 0\end{pmatrix}\
            \Big|\ g \in R\right\}\]
            is an ideal of $\ext{S}{R}$,
        \item $\m^2 = 0$ and the additive subgroups of $\m$ and $R$ are
            isomorphic,
        \item $\ext{S}{R}/\m  \simeq  S$, so $\ext{S}{R}$ is an extension of $\m$ by
            $S$.
    \end{enumerate}
\end{prop}

\paragraph{Remark}
    Let $R$ be a ring and $d:R\to R$ be a derivation. It is known
    (e.g. \cite{herstein}) that
    \[
    \left\{ \begin{pmatrix}f & d(f)\\0 & f\end{pmatrix}\
        \Big|\ f \in R\right\} \subseteq \M_2(R).
    \]
    is a subring of $\M_2(R)$ isomorphic to $R$; our construction is a generalization of this result.

\begin{cor}
    \label{commtononcomm}
    Let $R$ be a ring, $S$ be a central subring of $R$ and $d_1,d_2:S\to R$ be incomparable derivations such that
    $d_1(a)d_2(b) \neq 0$ whenever $d_1(a),d_2(b)\neq 0$. Then
    \begin{enumerate}
        \item $\ext{S}{R}$ is noncommutative.
        \item $\m$ is central in $\ext{S}{R}$ so $\ext{S}{R}$ is an extension of
            $\m$ by $S$ such that $\m\subseteq Z(\ext{S}{R})$.
    \end{enumerate}
\end{cor}

\begin{proof}
    \def\varphiX{
    \begin{pmatrix}
        x_1 & d_1(x_1) & 0\\
        0 & x_1 & 0\\
        0 & 0 & x_1\\
    \end{pmatrix}
    }
    \def\varphiY{
    \begin{pmatrix}
        x_2 & 0 & 0\\
        0 & x_2 & d_2(x_2)\\
        0 & 0 & x_2\\
    \end{pmatrix}
    }
    \def\varphiT{
    \begin{pmatrix}
        0 & 0 & d_1(x_1)d_2(x_2)\\
        0 & 0 & 0\\
        0 & 0 & 0\\
    \end{pmatrix}
    }
    If $x_1, x_2\in S$ are such that
    $d_1(x_1) \neq 0,d_1(x_2) = 0, d_2(x_1) = 0, d_2(x_2) \neq 0$, then
    \[
    \varphiX\cdot \varphiY - \varphiY \cdot \varphiX = \varphiT \neq 0.
    \]
    The fact the $\m$ is central is clear from the construction and
    centrality of $S$.
\end{proof}

The second statement of the corollary motivates the following definitions,
which are used for brevity and clarity
\begin{defn}
    Let $R$ be an extension of $I$ by $Q$. We will call it
    \begin{enumerate}
        \item \emph{ZC} (central by commutative) iff $I\subseteq
            Z(R)$ and $Q$ is commutative,
        \item \emph{NZC} iff it is ZC and noncommutative.
    \end{enumerate}
\end{defn}
Thus the thesis of \ref{commtononcomm} can be rephrased as ``$\ext{S}{R}$ is a
NZC of $\m$ by $S$''.\newline


Propositions \ref{mainconstruction},\ref{commtononcomm} show that one can apply incomparable
derivations to construct a NZC extension. The following proposition
gives a partial converse.

\begin{prop}
    Let $K$ be a field. Suppose that $R$ is an NZC extension of $I$ by $K$. Then there exist a unitary $(K, K)$--bimodule
    $M$, such that $k\cdot m = m\cdot k$ for all $k\in K, m\in M$,
    and two incomparable derivations from $K$ to $M$.
    \label{DoesNotExistThm}
\end{prop}
\begin{proof}
    \def\ad#1{\operatorname{ad}_{#1}}
    \def\dx{\tilde{\ad{x}}}
    \def\dy{\tilde{\ad{y}}}
    Let $[R, R]$ be the commutator of $R$, i.e. the smallest ideal of $R$ such
    that $R/[R, R]$ is commutative. Put $M = [R, R]$.
    Since $I$ is central in $R$, for every $i\in I, x, y\in R$ we have $ixy = (ix)y = yix = iyx$, so $i(xy
    - yx) = 0$, thus $I\cdot M = 0$. Similarly $M\cdot I = 0$, so $M$ has a
    $(K, K)$-bimodule structure induced from the $(R, R)$-bimodule structure.

    In \cite{puczylowski} Proposition 5.1.ii it was shown that the preimage in $R$ of
    $1\in K$ is contained in the center of $R$. It means that
    if $r\in R$ belongs to this preimage then
    \[r\cdot (ab - ba) = (ar)b - b(ar) = (a+i)b - b(a+i) = ab - ba\]
    where $i = ar-a \in I$ because the images of
    $a,ar$ in $K$ coincide. Applying centrality of $r$ once
    more we deduce that $M=[R, R]$ is a unitary $(K,K)$--bimodule.

    Since $M \subseteq I \subseteq Z(R)$ we have $r\cdot m = m\cdot r$ for all $r\in
    R,m\in [R, R]$, so $k\cdot m = m\cdot k$ for all $k\in K, m\in [R, R]$.

    For any $t\in R$ the map $\ad{t}:R \to R$
    defined by $\ad{t}(u):=tu-ut$ is in fact a derivation of $R$ to $[R,R]=M$.
    Since $\ad{t}(I) = 0$ we can define the derivation from $K$ to $M$ by $\tilde{\ad{t}}(r + I) := \ad{t}(r)$.
    Let $x, y\in R$ be elements which do not commute, then $\dx(x + I) = 0,
    \dx(y + I) \neq 0, \dy(x + I) \neq 0,
    \dy(y + I) = 0$, so $\dx, \dy: K\to M$ are incomparable.
\end{proof}

\section{Construction 2}

    \def\m{\mathfrak m}

In this section we give a specific construction  of an ideal extension of a
ring $I$ with $\ann{I}{I}\neq 0$ by a given field. Together with the
construction given in the previous section it allows us to solve
\sp{}. Roughly speaking, we construct a nonzero
    homomorphism from a extension of $\m$ by a field $K$ with very strong conditions on
    $\m$ to a extension of $I$ such that $\ann{I}{I} \neq 0$ by $K$. We
    will use this construction to produce a noncommutative extension of an
    ring $I$ satisfying $I^2 = I$ by a field $K$, but it can be applied
    to produce various other types of extensions by a field.

\begin{defn}
    \label{semidirect}
    If $Q$ is a ring and a ring $I$ is an $(Q, Q)$-bimodule such that the following
    \emph{compatibility conditions} are fulfilled
    \[q\cdot (i_1i_2) = (q\cdot i_1)i_2,\ (i_1\cdot q)i_2 =
    i_1(q\cdot i_2),\ (i_1i_2)\cdot q = i_1(i_2\cdot q)\ \ \forall q\in Q,
    i_1,i_2\in I\]
    then $R:=I\oplus_{Ab} Q$ is a ring with respect to the multiplication defined by
    \[
    (i_1, q_1)(i_2, q_2) := (i_1i_2 + i_1\cdot q_2 + q_1\cdot i_2, q_1q_2).
    \]
    If we identify $I$  with ideal $\{(i, 0)\ |\ i\in I\}$ of $R$ and $Q$ with
    $\{0\}\times Q \subseteq R$ then the obtained ring is an extension of $I$ by $Q$.
    We call it the \emph{semidirect extension} of $(Q,
    Q)$--bimodule $I$ by $Q$.
\end{defn}

Note that if $R$ is a semidirect extension of $I$ by $Q$ such that $I$ is central in $R$ and
$Q  \simeq R/I$ is commutative then directly from the definition of multiplication in
$R$ it follows that $R$ is commutative.

\begin{prop}
    \label{merging-lemma}
    Let $S$ be an extension of $\m$ by a field $K$ such that $0 \neq \m/\m^2 \subseteq Z(S/\m^2)$.

    Let $I$ be a $K$--algebra with $\ann{I}{I} := \left\{ i\in I\ |\ iI = Ii
    = 0\right\} \neq 0$.

    For any $m\in\m\setminus\m^2$ there exists a ring $T=T_{m}$ and
    a homomorphism $\varphi: S\to T$ such that
    \begin{enumerate}
        \item $T$ is an extension of $I$ by $K$ and $I$ is central in $T$
            (so $T$ is a ZC extension),
        \item $\varphi(m) \neq 0$,
        \item if $h:T  \rightonto T/I$ is the canonical epimorphism then $h\circ
            \varphi: S \to T/I$ is also an epimorphism.
    \end{enumerate}
\end{prop}

\begin{proof}
    \def\m{\mathfrak m}
    The epimorphism $f: S\rightonto S/\m  \simeq K$ induces a $(S,
    S)$--bimodule structure on $I$, given by $s\cdot i = i\cdot s = f(s)\cdot
    i$. This structure satisfies the compatibility
    conditions from \ref{semidirect}, so we can build an extension $R$ of $I$
    by $S$, which, as an abelian group, has the structure $I\oplus_{Ab} S$.

    Choose any $a\in \ann{I}{I}$. The quotient $J':=\m/\m^2$, containing $m
    \neq 0$, is a $K$-algebra, so we have a left $K$-module epimorphism $\psi': J' \rightonto Ka$ such that
    $\psi'(m) = a$, which is also a left $S$-module epimorphism.

    Let $\psi: \m  \rightonto Ka$ be the composition $\m \rightonto \m/\m^2
    \overset{\psi'}\rightonto Ka$. Let
    \[
    J:=\left\{ (\psi(j), j)\ |\ j\in \m \right\} \subseteq R.
    \]
    Then $J$ is an ideal of $R$, by direct verification (checking the right
    side we make use of the centrality of $\m/\m^2$). Moreover $J\cap I =
    0$ and $R/(I+J) \simeq S/( (I+J)\cap S) = S/\m  \simeq K$ so $R/J$ is an
    extension of $I$ by $K$.

    We set $T:=R/J$. The image of $I$ is central in $T$ because $I$ was central in
    $R$, so $T$ is a ZC extension of $I$ by $K$. Let $\varphi$ be the composition $S\to R \rightonto R/J=T$.
    In $T$ the image of $m$ is equal to the image
    of $0\neq -\psi(m)\in I$, which is nonzero by $I\cap J = 0$, so $\varphi(m)\neq 0$.
\end{proof}
\section{Proof of Theorem \ref{mainone}}

    Before we can prove the theorem \ref{mainone} we need one more ring
    example

    \begin{defn}
        The set $S=\{x^{\alpha}\ |\ \alpha\in (0, 1]\}$ is a semigroup with
        respect to the multiplication defined by
        \[x^{\alpha}\cdot x^{\beta} = \begin{cases}x^{\alpha + \beta} & \hbox{if }\alpha +
            \beta \leq 1\\ 0 &\hbox{otherwise.}\end{cases}\]
        The Zassenhaus algebra $Z = KS$ is defined as the semigroup $K$-algebra of $S$.

        This algebra satisfies $Z^2 = Z$ and has nonzero anihilator, equal to $Kx^1$.
    \end{defn}

    \begin{mainthm}
        For a given field $K$, the following conditions are equivalent
        \begin{enumerate}
            \item\label{mo1} there exists a field $L$ containing $K$ and two derivations
                from $K$ to $L$ whose kernels
                are incomparable with respect to inclusion,
            \item\label{mo2} there exists a noncommutative ring $R$ and its ideal $I$ such that $I$ is
                central in $R$, $R/I \simeq K$ and $I^2 = 0$,
            \item\label{mo3} there exists a noncommutative ring $R$ and its ideal $I$ such that $I$ is
                commutative as a ring, $I =
                I^2$ and $R/I \simeq K$.
            \item\label{mo4} there exists a noncommutative ring $R$ and its ideal $I$ such that $I$ is
                central in $R$ and $R/I \simeq K$.
        \end{enumerate}
    \end{mainthm}

    \begin{proof}
        $\ref{mo1}.\implies\ref{mo2}.$ Let $R = \ext{K}{L}$ be the ring obtained
        from construction \ref{extconstruction} applied to subring $K$ of ring $L$ and two
        derivations given. Corollary \ref{commtononcomm} says that
        $R$ is a noncommutative extension of $\m$ by $K$, $\m$ is central in
        $R$ and $\m^2 = 0$, so $R$ proves \ref{mo2}.

        $\ref{mo2}.\implies\ref{mo3}.$
        Let $R, I$ be the ring and ideal from \ref{mo2} and
        $Z$ be the Zassenhaus algebra over $K$.
        Choose $x,y\in R$ such that $xy-yx\neq 0$ and put $m:=xy-yx$,
        then $R, I, Z, m$ satisfy the conditions
        for $S, \m, I, m$ of proposition \ref{merging-lemma}.
        If $T:=T_{m}$ the extension obtained in \ref{merging-lemma} and $\varphi: S\to
        T$ is the corresponding
        homomorphism, then $0\neq \varphi(xy-yx) = \varphi(x)\varphi(y) -
        \varphi(y)\varphi(x)$, so $T$ is a
        noncommutative extension of $Z = Z^2$ by $R/I  \simeq K$.

        $\ref{mo3}.\implies\ref{mo4}.$ If $i_1, i_2\in I, r\in
        R$ then $i_1i_2r = i_1(i_2r) = (i_2r)i_1 = i_2(ri_1) = ri_1i_2$, so
        $I^2 = I$ is central in $R$ and the implication is obvious.

        \def\I{\mathcal{I}}
        $\ref{mo4}.\implies\ref{mo1}.$ The ring $R$ from \ref{mo4} is a NZC by
        $K$, so we can apply proposition \ref{DoesNotExistThm}.

        Let $M$ and $d, d':K\to M$ be the $(K, K)$--bimodule and incomparable derivations
         resulting from applying \ref{DoesNotExistThm}. As $M$ is unitary, we can view
            $M$ as an (left) $K$ vector space. We take a direct sum of it with
            another $K$-vector space obtaining an infinitely dimensional
            vector space $M'$ having basis
            $\I = (x_i)_{i\in \I}$.

            Let $L:=K\left(\{y_{i}\}_{i\in \I}\right)$ be the purely
            transcendental extension of $K$. As $L$ has a basis of cardinality
            $(\card \I)\cdot \aleph_0 = \card \I$ we can take
            a bijection of $\I$ with this basis and extend it to a $K$-linear
            isomorphism $M' \to L$.
            As $k\cdot l = l\cdot k$ and $k\cdot m =
            m\cdot k$ for any $k\in K, l\in L, m\in M'$, the isomorphism preserves
            the $(K, K)$--bimodule structure, so compositions of $d_1, d_2$
            which $M \rightinto M' \to L$ are incomparable derivations from $K$ to field $L$.
    \end{proof}

\begin{cor}
    Let $K$ be an algebraic extension of a prime field. Then there does not
    exists a NZC by $K$.
\end{cor}

\begin{proof}
    This follows directly from theorem \ref{mainone} and corollary \ref{incomparabledontexists}.
\end{proof}

\section{Comments}

If $K$ is a field and $d_1,d_2:K\to K$ are incomparable derivations,
then the ring $\ext{K}{K}$ has only 3 (left or right
or both-sided) ideals: $\ext{K}{K}, \m$ and $0$. It has an unity, thus it is local and left (and right)
artinian of length $2$. This shows a similarity of $\ext{K}{K}$ to the ring
$K^0$ -- the abelian group $(K, +)$ equipped with zero multiplication -- with
unity adjoined.

\paragraph{Another application of construction \ref{mainconstruction},}
mentioned in the introduction.

\begin{prop}
    Let $k$ be a field and $R$ be the group algebra $k\ring{x, y}$ of the free
    group $\gen{x, y}$. The commutator ideal $(xy-yx)$ of $R$
    does not satisfy $(xy-yx)^2 = (xy-yx)$.
\end{prop}
\begin{proof}
Let $k$ be a field, $G = \gen{x, y}$ be the free group generated by two
elements and $R$ be the group algebra of the group $G$. Let
$R_{ab}:= R/(xy-yx)$ be the group algebra of the free abelian group.

Let $d_1 = \frac{\partial}{\partial x}, d_2 = \frac{\partial}{\partial y}$ be the
derivations of $R$. They preserve $(xy-yx)$, to they induce derivations of
$R_{ab}$, denoted for simplicity $d_1,d_2$.
\def\varphiX{
\begin{pmatrix}
    x & d_1(x) & 0\\
    0 & x & 0\\
    0 & 0 & x\\
\end{pmatrix}
}%
\def\varphiY{
\begin{pmatrix}
    y & 0 & 0\\
    0 & y & d_2(y)\\
    0 & 0 & y\\
\end{pmatrix}
}%
Let $(\ext{S}{R})_{ab}$ be the extension resulting from applying \ref{mainconstruction} to
$S=R=R_{ab}, d_1, d_2$.

Since $\varphiX, \varphiY$ are invertible in
$(\ext{S}{R})_{ab}$ we can construct a homomorphism
\[
\varphi: R \to (\ext{S}{R})_{ab}.
\]

\def\varphiI{
\begin{pmatrix}
    0 & 0 & 1\\
    0 & 0 & 0\\
    0 & 0 & 0\\
\end{pmatrix}
}
We calculate that $\varphi(xy - yx) = \varphiI$, so the image of the ideal
$(xy-yx)$
lies in ideal with zero multiplication, so $(xy-yx)$ does not satisfy
$(xy-yx)^2 = (xy-yx)$.
\end{proof}

Similarly, using more that one derivation, one obtains
\begin{cor}
    Let $k$ be a field and $R$ be the group algebra $k\ring{x, y}$ of the free
    group $\gen{x, y}$. The ideals $I = (xy-yx), I^2,
    I^3, \dots$ of $R$ are all distinct.
\end{cor}

\section{Acknowledgements}
    The author expresses his great gratitude to Professor Edmund R. Puczylowski for
    his help during the preparation of the paper, especially his patient
    verification of numerous drafts.

\end{document}